\documentclass[12pt,reqno]{amsart}
\usepackage{graphicx}
\usepackage{amssymb,amsmath}
\usepackage{amsthm}
\usepackage{color}
\usepackage[pdf]{pstricks}
\usepackage{hyperref}
\usepackage{empheq}
\usepackage{pgfplots}

\usepackage{mathtools}

\usepackage[normalem]{ulem}

\usepackage{tikz}
\usepackage{graphicx}  
\usepackage{amsmath}  
\usepackage{amssymb,amsthm}
\usepackage{color}

\newcommand{\R}{{\mathbb R}}

\newtheorem{theorem}{Theorem}
\newtheorem{lemma}{Lemma}

\newtheorem{proposition}{Proposition}

\begin{document}

\title[Orbital stability of kinks]{\bf Orbital stability of kinks \\ in the NLS equation with competing nonlinearities}

\author[J. Holmer]{Justin Holmer}
\address[J. Holmer]{Department of Mathematics, Brown University, Providence, RI 02912, USA}

\author[P. G. Kevrekidis]{Panayotis G. Kevrekidis}
\address[P. G. Kevrekidis]{Department of Mathematics and Statistics, University of Massachusetts Amherst, Amherst, MA 01003-4515, USA}
\address[P. G. Kevrekidis]{Department of Physics, University of Massachusetts Amherst, Amherst, 01003, MA, USA}
\address[P. G. Kevrekidis]{Theoretical Sciences Visiting Program, Okinawa Institute of Science and Technology Graduate University, Onna, 904-0495, Japan}
\email{kevrekid@umass.edu}

\author[D. E. Pelinovsky]{Dmitry E. Pelinovsky}
\address[D. E. Pelinovsky]{Department of Mathematics and Statistics, McMaster University, Hamilton, Ontario, Canada, L8S 4K1}
\email{pelinod@mcmaster.ca}

\maketitle

\begin{abstract} 
	Kinks connecting zero and nonzero equilibria in the NLS equation with competing nonlinearities occur at the special values of the frequency parameter. Since they are minimizers of energy, they are expected to be orbitally stable in the time evolution of the NLS equation. However, the stability proof is complicated by the degeneracy of kinks near the nonzero equilibrium. The main purpose of this work is to give a rigorous proof of the orbital stability of kinks. We give details of analysis for the cubic--quintic NLS equation and show how the proof is extended to the general case. 
\end{abstract}

\section{Introduction}
\label{sec-1}

\subsection{Background and motivations}
\label{sec-1-1}

Dispersive models with competing nonlinearities
have been of considerable interest over the
past few years. One of the prominent examples
of this type concerns the so-called cubic-quintic
nonlinearity~\cite{Baldelli2025NoDEA,crasovan}, which is not only
of considerable mathematical interest~\cite{KillipOhPocovnicuVisan2017}, but
also of practical applications in physical 
experiments in nonlinear optics~\cite{FalcaoFilho2013PRL,Reyna2020PRA}.

The study of ultracold atomic
gases has provided further motivation for
the study of such competing focusing and
defocusing nonlinear terms. This concerns
the role of quantum fluctuations competing with mean-field effects, as has
been shown in gases of $^{39}$K~\cite{Luo2021FrontPhys,Mistakidis2023PhysRep}. 
The theoretical works of~\cite{Petrov2015,PetrovAstrakharchik2016},
have identified different types of competing
nonlinear contributions in different dimensions:
a quadratic focusing and a cubic defocusing
in 1D, a logarithmic term encompassing both effects in 
2D and a cubic-quartic model in 3D.

Models with competing nonlinearities have a special coherent structure 
called {\em the kink}, which connects zero and nonzero equilibria \cite{mistakidis2024generic}, 
see also \cite{Susanto2025} in the discrete setting. 
As the frequency parameter is varied, such systems possess two 
monoparametric families of solutions, i.e.,
the so-called droplets (homoclinic to a vanishing background)~\cite{Luo2021FrontPhys}
and the so-called bubbles (homoclinic to a non-vanishing background)~\cite{Katsimiga2023PRA063308}:
the kink arises at the edge between the existence regions 
of these two families and for a {\it unique} value of the frequency. The kink 
behaves partly as a dark soliton due to its heteroclinic form
and partly as a bright soliton as one
of its asymptotes is to a vanishing state. 
This poses a fundamental challenge from 
a mathematical perspective regarding how
to handle the rigorous stability analysis
of the kink. It is the scope
of the present work to establish this
fundamental result, initially for the cubic-quintic case, and subsequently for a general NLS equation with competing nonlinearities. 
Nonlinear stability of kinks in higher spatial dimensions with respect to transversely periodic perturbations, 
suggested in \cite{mistakidis2024generic}, follows from the nonlinear stability of kinks in one spatial dimension.

\subsection{Main result}
\label{sec-1-2}

To simplify the presentation, we formulate the main result for the cubic--quintic NLS equation written in the normalized form
\begin{equation}
\label{NLS}
i \psi_t = \psi_{xx} - \psi + 4 |\psi|^2 \psi - 3 |\psi|^4 \psi.
\end{equation}
The coefficients of the focusing cubic and defocusing quintic nonlinearities  in (\ref{NLS}) have been normalized without loss of generality
due to the scaling transformation, 
which, in turn, leads ---as discussed, e.g., in~\cite{mistakidis2024generic} via
a potential energy landscape analysis---
to the unique selection of the frequency for the kink 
solutions of the cubic--quintic NLS equation, normalized to $1$ in (\ref{NLS}) for convenience.

The cubic--quintic NLS equation (\ref{NLS}) is associated variationally 
with the conserved energy function 
\begin{equation}
\label{energy}
E(\psi) = \int_{\mathbb{R}} \left[ |\psi_x|^2 + |\psi|^2 (1 - |\psi|^2)^2 \right] dx,
\end{equation}
which is defined in the energy space 
\begin{equation}
\label{energy-space}
\mathcal{E} = \{ \psi \in H^1_{\rm loc}(\R) : \;\; 
\psi_x \in L^2(\R), \;\; \psi (1-|\psi|^2) \in L^2(\R)\}.
\end{equation}
The NLS equation (\ref{NLS}) is locally and globally well-posed in $\mathcal{E}$, see Appendix \ref{S:lwp} for a short proof.
Critical points of $E(\psi)$, denoted by $\phi$, satisfy the stationary NLS equation
\begin{equation}
\label{ode}
\phi'' - \phi + 4 |\phi|^2 \phi - 3 |\phi|^4 \phi = 0.
\end{equation}
If $\phi = \phi(x)$ is real-valued, then solutions of the stationary NLS equation (\ref{ode}) are given by the level curves of the first-order invariant $I(\phi,\phi') = (\phi')^2 - \phi^2 (1-\phi^2)^2$ on the phase plane $(\phi,\phi')$. The level $I(\phi,\phi') = 0$ contains a pair of heteroclinic obits from $(0,0)$ to $(1,0)$ and from $(1,0)$ to $(0,0)$ which are referred to as the kinks. In particular, we shall consider the monotonically increasing 
profile $\phi(x) : \mathbb{R} \to \mathbb{R}$ obtained from 
\begin{equation}
\label{first-order}
\phi' = \phi (1-\phi^2) > 0,
\end{equation}
such that $\phi(x) \to 0$ as $x \to -\infty$ and $\phi(x) \to 1$ as $x \to +\infty$. The exact solution to (\ref{first-order}) is
\begin{equation}
    \label{exact-kink}
    \phi(x) = \left( \frac{1}{2} (1 + \tanh(x)) \right)^{1/2},
\end{equation}
where we have set the translational parameter to $0$ without loss of generality.

The kink with the profile $\phi$ is spectrally stable because 
the linearized equations of motion are defined by a pair of self-adjoint
Schr\"{o}dinger operators $\mathcal{L}_{\pm} : H^2(\R) \subset L^2(\R) \to L^2(\R)$ which are positive-definite. Indeed, linearization of 
the NLS equation (\ref{NLS}) about the real-valued profile $\phi$ with the complex-valued perturbation $u+iv$ gives the linearized equations 
\begin{equation}
\label{lin-NLS}
\begin{cases}
u_t = -\mathcal{L}_- v, \\ 
v_t = \mathcal{L}_+ u,
\end{cases} 
\qquad 
\begin{cases}
\mathcal{L}_- = -\partial_x^2 + 1 - 4 \phi^2 + 3 \phi^4, \\
\mathcal{L}_+ = -\partial_x^2 + 1 - 12 \phi^2 + 15 \phi^4.
\end{cases}
\end{equation}
Since $\mathcal{L}_- \phi = 0$ and $\phi(x) \geq 0$ for all $x \in \mathbb{R}$, 
the spectrum of $\mathcal{L}_-$ in $L^2(\mathbb{R})$ is non-negative. 
Since $\mathcal{L}_+ \phi' = 0$ and $\phi'(x) \geq 0$ for all $x \in \mathbb{R}$, the spectrum of $\mathcal{L}_+$ in $L^2(\mathbb{R})$ is 
also non-negative. Hence, $\phi$ is a minimizer of the energy $E(\psi)$. 
In agreement with a general theory of spectral stability \cite[Section 3.5]{GP25}, 
the spectrum of the linearized operator 
\begin{equation}
\label{operators}
\left( \begin{matrix} 0 & -\mathcal{L}_- \\
\mathcal{L}_+ & 0 \end{matrix} \right) : H^2(\R) \times H^2(\R) 
\subset L^2(\R) \times L^2(\R) \to L^2(\R) \times L^2(\R)
\end{equation}
belongs to $i \mathbb{R}$, as is illustrated in Figure \ref{fig:depjustin_f1} by using the 
finite-difference approximations.

\begin{figure}[htbp]
	\includegraphics[width=0.5\textwidth,height=0.25\textheight]{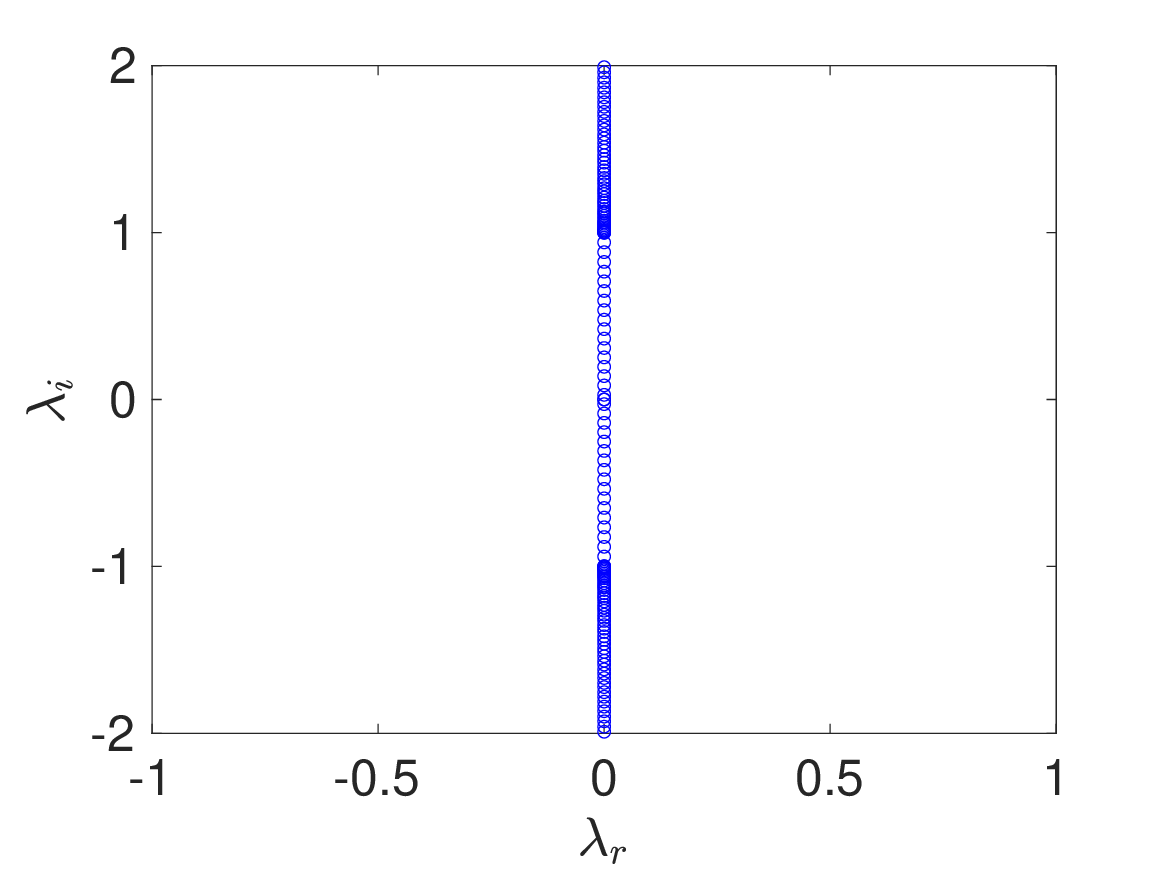}
	\caption{Spectrum of the linearized operator (\ref{operators}) in the complex $\lambda$-plane, where $\lambda=\lambda_r + i \lambda_i$.}
	\label{fig:depjustin_f1}
\end{figure}

The minimizer of energy $E(\psi)$ with the profile $\phi$ is degenerate due to symmetries of the NLS equation (\ref{NLS}) and due to the absence of the spectral gap near the zero eigenvalue. This is seen in Figure \ref{fig:depjustin_f1} because of the nonzero boundary condition $\phi(x) \to 1$ as $x \to +\infty$. The proof of its orbital stability is not a trivial adaptation of the Lyapunov stability theory, see \cite[Chapter 2]{GP25} for a review. Hence, we address 
the main question on how to prove the orbital stability of the kink with respect to perturbations in the energy space $\mathcal{E}$ in (\ref{energy-space}).

Let us introduce the Hilbert space $\mathcal{H}_R$ for a given $R > 0$ according to the inner product 
\begin{equation}
\label{Hilbert-space}
\langle f, g \rangle_{\mathcal{H}_R} := \int_{-\infty}^R f(x) \bar{g}(x) dx + \int_R^{\infty} \frac{1-\phi^2(x)}{1 - \phi^2(R)} f(x) \bar{g}(x) dx
\end{equation}
and the induced norm $\| f \|_{\mathcal{H}_R} = \left( \langle f, f \rangle_{\mathcal{H}_R} \right)^{1/2}$. The form and operator domains
for the linear operators $\mathcal{L}_+$ and $\mathcal{L}_-$ extended to $\mathcal{H}_R$ are defined as $\mathcal{H}^1_R(\R) = \dot{H}^1(\R) \cap \mathcal{H}_R$ and $\mathcal{H}^2_R(\R) = \dot{H}^2(\R) \cap \mathcal{H}_R$ respectively. We also define 
the distance between two elements in $\mathcal{E} \cap \mathcal{H}_R$ as
\begin{equation}
\label{distance}
\rho_R(\psi,\phi) =  \| \psi' - \phi'\|^2_{L^2(\R)} + 
\| \psi - \phi \|_{\mathcal{H}_R}^2 + \| |\psi|^2 - |\phi|^2 \|^2_{L^2(R,\infty)}. 
\end{equation}
The following theorem represents the main result of this study.

\begin{theorem}
	\label{th-main}
	Fix $R > 0$ sufficiently large such that $\phi(R) > \frac{\sqrt{2}}{\sqrt{3}}$. The kink with the profile $\phi$ is orbitally stable in the NLS equation (\ref{NLS}) with respect to perturbations  in $\mathcal{E} \cap \mathcal{H}_R$. To be precise, for every small $\epsilon > 0$ there is $\delta > 0$ such that if $\psi_0 \in \mathcal{E}$ satisfies  $\rho_R(\psi_0,\phi) \leq \delta^2$, then the unique solution $\psi \in C^0(\R,\mathcal{E})$ to the NLS equation (\ref{NLS}) with $\psi |_{t = 0} = \psi_0$ satisfies 
	\begin{equation}
	\label{bound-distance}
	\inf_{\alpha,\beta \in \R} \rho_R\left(e^{i \alpha} \psi(t,\cdot + \beta),\phi \right) \leq \epsilon^2, \quad t \in \R.
	\end{equation}
\end{theorem}

\subsection{Methodology and organization of the paper}
\label{sec-1-3}

Comparing (\ref{energy-space}) with (\ref{Hilbert-space}) and (\ref{distance}) shows that the only difference between the distance in $\mathcal{E}$ and $\rho_R$ is the exponential weight added on $[R,\infty)$. A similar technique to deal with the orbital and asymptotic stability of the black soliton 
in the cubic NLS equation was proposed in \cite{GS}, after the previous studies of their linear and orbital stability in \cite{Saut,Chiron,Gallo}. This technique was further extended to the black soliton for the cubic NLS equation in spaces of higher regularity \cite{GPII}, the quintic NLS equation \cite{Alejo}, and the NLS equation with the intensity--dependent dispersion \cite{PP24}. In all cases, the black soliton is a saddle point of energy with a simple negative eigenvalue in the linearized operator $\mathcal{L}_-$. However, the symmetry-induced constraints allow us to obtain coercivity of the energy functional in the weighted Sobolev spaces and to prove that the black soliton is a constrained minimizer of energy. See also \cite[Chapter 7]{GP25}.

Compared to the case of the black solitons, the kink is already a minimizer of energy. The symmetry-induced constraints are sufficient to show 
coercivity of the quadratic part of the energy functional in $\mathcal{H}_R^1$. However, the anharmonic part of the energy functional is not positive definite 
unless the value of $R > 0$ in $\mathcal{H}_R$ is chosen from the condition $\phi(R) > \frac{\sqrt{2}}{\sqrt{3}}$. In addition, partitioning of $\mathbb{R}$ into $(-\infty,R)$ and $(R,\infty)$ induces a negative eigenvalue in the corresponding linearized operators and we need to choose $R > 0$ sufficiently large to rigorously prove that the symmetry-induced constraints are sufficient for coercivity of the quadratic part. This leads to the proof of Theorem \ref{th-main} developed in Section \ref{sec-2}.

The case of kinks is similar to the case of domain walls in the coupled system of cubic NLS equations considered in \cite{CPP18}. However, the orbital stability result for domain walls was defined in the case of $R \gg 1$ by the price that the perturbations are exponentially small with respect to $R$. 
Compared to the domain walls, we are not using here the asymptotic limit of 
large $R$ and hence the perturbations in $\rho_R$ are not exponentially small in $R$. Moreover, the numerical approximations developed in Section \ref{sec-3} suggest that the value of $R > 0$ can be chosen arbitrarily as long as $\phi(R) > \frac{\sqrt{2}}{\sqrt{3}}$. 

The proof of the orbital stability of kinks can be easily extended to the general case of competing power nonlinearities at and above the cubic nonlinearity. We give some details in Section \ref{sec-4} to show that
the expansion of energy developed for the cubic-quintic NLS equation (\ref{NLS}) is robust enough to handle the general case. 

The local well-posedness of the initial-value problem for (\ref{NLS}) in the energy space $\mathcal{E}$ is a necessary ingredient to the orbital stability theory. Although the proof of local  well-posedness in $\mathcal{E}$ is rather standard, based on \cite{Zhidkov2001}, we give it for consistency 
in Appendix \ref{S:lwp}.

\section{Proof of the orbital stability of kinks}
\label{sec-2}

\subsection{Expansion of the energy}
\label{sec-2-1}

Following \cite{GS} (see also adaptations of this method in 
\cite{Alejo,CPP18,GPII,PP24}), we introduce the variable 
\begin{equation}
\label{eta-variable}
\eta := |\phi + u + iv|^2 - \phi^2 = 2 \phi u + u^2 + v^2.
\end{equation}
Since $\phi$ is a critical point of $E(\psi)$, we have the energy 
decomposition 
\begin{equation}
\label{energy-decomp}
(\delta^2 E)(u,v) = \int_{\R} \left[ 
(u')^2 + (v')^2 + (1- \phi^2) (1-3 \phi^2) (u^2 + v^2) + (3 \phi^2 - 2) \eta^2 + \eta^3 \right] dx,
\end{equation}
where $(\delta^2 E)(u,v) = E(\phi + u + iv) - E(\phi)$. This can be rewritten by using the quadratic form associated with the operator 
$\mathcal{L}_-$ in (\ref{lin-NLS}) as 
\begin{equation}
\label{energy-equiv}
(\delta^2 E)(u,v) = Q_-(u) + Q_-(v) 
+ \int_{\R} \left[ (3 \phi^2 - 2) \eta^2 + \eta^3 \right] dx,
\end{equation}
where 
$$
Q_-(u) = \int_{\R} \left[ 
(u')^2  + (1- \phi^2) (1-3 \phi^2) u^2 \right] dx.
$$
The quadratic form $Q_-(u)$ is positive definite and can be made coercive in an exponentially weighted space under a scalar orthogonality 
condition on $u$. However, $3 \phi^2 - 2$ is sign-indefinite on $\R$, hence coercivity 
on $\eta$ is lost if we use the decomposition (\ref{energy-equiv}) uniformly on $\R$.  To deal with the problem, let us fix $R \in \R$ such that 
$\phi(R) > \frac{\sqrt{2}}{\sqrt{3}}$ and use the decomposition 
\begin{align}
(\delta^2 E)(u,v) &= \int_{-\infty}^R \left[ q_+(u) + q_-(v) + (3 \phi^2 -2) (u^2 + v^2) (4 \phi u + u^2 + v^2) + (2 \phi u + u^2 + v^2)^3 \right] dx 
\notag \\
& \qquad + \int_{R}^{+\infty} \left[ q_-(u) + q_-(v) + (3 \phi^2 - 2) \eta^2 + \eta^3 \right] dx,
\label{energy-correct}
\end{align}
where $q_{\pm}$ is a density for the quadratic forms $Q_{\pm}$ associated with the linearized operators $\mathcal{L}_{\pm}$ in (\ref{lin-NLS}). The quadratic form for $v$ in (\ref{energy-correct}) is the same as $Q_-(v)$ and it is associated with the operator 
$\mathcal{L}_- : H^2(\R) \subset L^2(\R) \to L^2(\R)$. On the other hand, the quadratic form for $u$ in (\ref{energy-correct}) is associated with the operator 
$\mathcal{L}_R : H^2(\R) \subset L^2(\R) \to L^2(\R)$ given by 
\begin{equation}
\label{L-R}
\mathcal{L}_R = -\partial_x^2 + V_R, \quad V_R(x) = \left\{ 
\begin{array}{ll} 
1 - 12 \phi^2(x) +  15 \phi^4(x), \quad & x \in (-\infty,R), \\
1 - 4 \phi^2(x) + 3 \phi^4(x), \quad & x \in (R,\infty).
\end{array}
\right.
\end{equation}
The potential $V_R$ in (\ref{L-R}) is piecewisely smooth, bounded at infinity, with a bounded jump at $x = R$. 

\subsection{Spectral estimates}
\label{sec-2-2}

Let us first discuss the spectra of the self-adjoint operators $\mathcal{L}_+, \mathcal{L}_-, \mathcal{L}_R$ in $L^2(\R)$ 
with the domain in $H^2(\R)$. These results follows from the standard Weyl's and Sturm's theory for the Schr\"{o}dinger operators. 
\begin{itemize}
  \item Since $1 - 12 \phi^2 + 15 \phi^4 \to 1$ as $x \to -\infty$ and $1 - 12 \phi^2 + 15 \phi^4 \to 4$ as $x \to +\infty$, 
  the spectrum of $\mathcal{L}_+$ consists of two branches of the continuous spectrum $[1,\infty)$ and $[4,\infty)$, a simple zero 
  eigenvalue with the eigenfunction $\phi'$, and (possibly) isolated simple eigenvalues in $(0,1)$. 

  \item Since $1 - 4 \phi^2 + 3 \phi^4 \to 1$ as $x \to -\infty$ and $1 - 4 \phi^2 + 3 \phi^4 \to 0$ as $x \to +\infty$,  
  the spectrum of $\mathcal{L}_-$ consists of two branches of the continuous spectrum $[1,\infty)$ and $[0,\infty)$.

  \item Since $V_R \to 1$ as $x \to -\infty$ and $V_R \to 0$ as $x \to +\infty$,  the spectrum of $\mathcal{L}_R$ consists of two branches of the continuous spectrum $[1,\infty)$ and $[0,\infty)$ and (possibly) isolated eigenvalues of finite algebraic multiplicity in $(-\infty,0)$.
\end{itemize}
In what follows, we analyze the self-adjoint operators $\mathcal{L}_-$ and $\mathcal{L}_R$ in a weighted $L^2(\R)$ space $\mathcal{H}_R$ given by (\ref{Hilbert-space}) with the corresponding domain $\mathcal{H}^2_R$. 
We establish the following coercivity estimates in $\mathcal{H}^1_R$ 
for the two quadratic forms $Q_-(v)$ and $Q_R(u)$ appearing in (\ref{energy-correct}). 

\begin{lemma}
	\label{lem-operator-minus}
For every $R \in \mathbb{R}$, there exists a positive constant $C_-$ such that 
	\begin{align}
Q_-(v) \geq C_- \left( \| v' \|^2_{L^2(\R)} + \| v \|^2_{\mathcal{H}_R}\right), \quad \forall v \in \mathcal{H}^1_R(\R) : \;\;  \langle v, \phi \rangle_{\mathcal{H}_R} = 0.
	\label{coercivity-minus}
	\end{align}
\end{lemma}

\begin{proof}
Consider the spectrum of $\mathcal{L}_- : \mathcal{H}^2_R(\R) \subset \mathcal{H}_R \to \mathcal{H}_R$, 
which can be obtained from the spectral problem 
	\begin{equation}
	\label{spectral-problem-minus}
		\mathcal{L}_- v = \lambda W_R v, \quad v \in \mathcal{H}^2_R(\R), \quad 
		W_R(x) = \left\{ 
		\begin{array}{ll} 
		1, \quad & x \in (-\infty,R), \\
		\frac{1 - \phi^2(x)}{1 - \phi^2(R)}, \quad & x \in (R,\infty),
		\end{array}
		\right.
	\end{equation}
where the weight function $W_R$ is due to the weight in the Hilbert space $\mathcal{H}_R$ given by (\ref{Hilbert-space}). Since $1 - 4 \phi^2 + 3 \phi^4 \to 1$ and $W_R \to 1$ as $x \to -\infty$, the spectrum of $\mathcal{L}_-$ in $\mathcal{H}_R$ includes the continuous spectrum $[1,\infty)$. 

Next, we show that there are only simple isolated eigenvalues in $(-\infty,1)$. 
Indeed, for every $\lambda \in (-\infty,-1)$, solutions to the second-order differential equation in (\ref{spectral-problem-minus}) are spanned 
by $\{ e^{\sqrt{1-\lambda}x}, e^{-\sqrt{1-\lambda}x} \}$ as $x \to -\infty$ and by $\{ 1,x \}$ as $x \to +\infty$. 
The eigenfunction $u \in \mathcal{H}^2_R(\R)$ corresponds to the connection between the decaying solution $e^{\sqrt{1-\lambda}x} \to 0$ 
as $x \to -\infty$ and the bounded solution $1$ as $x \to +\infty$ since $e^{-\sqrt{1-\lambda}x}$ as $x \to -\infty$ and $x$ as $x \to +\infty$ 
are not admissible in $\mathcal{H}_R^2(\R)$. Hence, there is no continuous spectrum in $(-\infty,1)$ and 
isolated eigenvalues are at most simple. 

Since $\mathcal{L}_- \phi = 0$ in $L^2(\R)$, $\phi \in \mathcal{H}^2_R$, and $\phi(x) > 0$ for all $x \in \mathbb{R}$, 
no eigenvalues exist in $(-\infty,0)$ due to the Rayleigh quotient 
$$
\inf_{v \in \mathcal{H}^1_R(\R)} \frac{Q_-(v)}{\| v \|^2_{\mathcal{H}_R}} = 0.
$$
The constraint $\langle v, \phi \rangle_{\mathcal{H}_R} = 0$ removes the lowest eigenvalue of $\mathcal{L}_-$ in $\mathcal{H}_R$. 
Since all isolated eigenvalues in $(-\infty,1)$ are simple, we get the coercivity estimate from the spectral theory:
$$
Q_-(v) \geq \lambda_0 \| v \|^2_{\mathcal{H}_R}, \qquad \forall v \in \mathcal{H}^1_R(\R) : \;\;  \langle v, \phi \rangle_{\mathcal{H}_R} = 0,
$$
where $\lambda_0$ is either the second eigenvalue of (\ref{spectral-problem-minus}) in $(0,1)$ or $1$ if no other eigenvalues of 
(\ref{spectral-problem-minus}) in $(0,1)$ exists. The coercivity estimate (\ref{coercivity-minus}) follows 
by G\"{a}rding's inequality.
\end{proof}

\begin{lemma}
	\label{lem-operator-R}
There is a sufficiently large $R_0 > 0$ such that for every $R \in (R_0,\infty)$, there exists a positive constant $C_R$ such that 
	\begin{align}
Q_R(u) := \int_{-\infty}^R q_+(u) dx + \int_{R}^{+\infty} q_-(u) dx &\geq C_R \left( \| u' \|^2_{L^2(\R)} + \| u \|^2_{\mathcal{H}_R}\right), \notag \\ 
& \qquad \forall u \in \mathcal{H}^1_R(\R) : \;\;  \langle u, \phi' \rangle_{L^2(\R)} = 0.
	\label{coercivity-R}
	\end{align}
\end{lemma}

\begin{proof}
By the same argument as in the proof of Lemma \ref{lem-operator-minus}, the spectrum of 
$\mathcal{L}_R : \mathcal{H}^2_R(\R) \subset \mathcal{H}_R \to \mathcal{H}_R$ includes the continuous spectrum $[1,\infty)$ 
and simple isolated eigenvalues in $(-\infty,1)$. Eigenvalues 
of $\mathcal{L}_R$ in $\mathcal{H}_R$ are defined from solutions of the spectral problem 
	\begin{equation}
	\label{spectral-problem}
		\mathcal{L}_R u = \lambda W_R u, \quad u \in \mathcal{H}^2_R(\R), \quad 
		W_R(x) = \left\{ 
		\begin{array}{ll} 
		1, \quad & x \in (-\infty,R), \\
		\frac{1 - \phi^2(x)}{1 - \phi^2(R)}, \quad & x \in (R,\infty),
		\end{array}
		\right.
	\end{equation}
for $\lambda \in (-\infty,1)$. Let $\phi_R \in \mathcal{H}^1_R(\R)$ be given by 
	$$
	\phi_R(x) = \left\{ 
	\begin{array}{ll} 
	\phi'(x), \quad & x \in (-\infty,R), \\
	\frac{\phi'(R)}{\phi(R)} \phi(x), \quad & x \in (R,\infty).
	\end{array}
	\right.
	$$
Since $\mathcal{L}_R \phi_R = 0$ in $H^{-1}(\R)$, we have $Q_R(\phi_R) = 0$. However, 
since $\phi_R \notin \mathcal{H}^2_R$, the spectral problem (\ref{spectral-problem}) may have 
eigenvalues $\lambda$ below $0$. 

We have $Q_R(u) \to Q_+(u)$ as $R \to \infty$ for every $u \in H^1(\mathbb{R})$ with $Q_+(u) \geq 0$ 
for every $u \in H^1(\mathbb{R})$ since $\mathcal{L}_+ \phi' = 0$ with 
$\phi'(x) > 0$ for all $x \in \mathbb{R}$. This implies that there exists only one simple eigenvalue of 
the spectral problem (\ref{spectral-problem}) near $0$ for sufficiently large $R$. 
We now add the constraint $\langle u, \phi' \rangle_{L^2(\R)} = 0$ to show coercivity 
(\ref{coercivity-R}). By using (\ref{first-order}), we have 
$$
\langle u, \phi_R \rangle_{\mathcal{H}_R} = \int_{-\infty}^R u(x) \phi'(x) dx + 
\frac{\phi'(R)}{\phi(R)} \int_{R}^{+\infty} u(x) \frac{1-\phi^2(x)}{1 - \phi^2(R)} \phi(x) dx = \langle u, \phi' \rangle_{L^2(\R)}.
$$
Let $\Pi_R : \mathcal{H}_R \to \mathcal{H}_R |_{\{ \phi_R \}^{\perp}}$ be the orthogonal projection operator given by 
$$
\Pi_R u = u - \frac{\langle u, \phi_R \rangle_{\mathcal{H}_R}}{\| \phi_R \|^2_{\mathcal{H}_R}} \phi_R, \quad \forall u \in \mathcal{H}_R.
$$
Eigenvalues of $\mathcal{L}_R |_{\{ \phi_R \}^{\perp}} : \mathcal{H}^2_R |_{\{ \phi_R \}^{\perp}} \subset \mathcal{H}_R |_{\{ \phi_R \}^{\perp}} \to \mathcal{H}_R |_{\{ \phi_R \}^{\perp}}$ are determined by eigenvalues of the constrained eigenvalue problem 
\begin{equation}
    \label{spectral-constrained}
\left( \mathcal{L}_R - \lambda W_R \right) u = \nu W_R \phi_R, \quad \langle u, \phi_R \rangle_{\mathcal{H}_R} = 0,
\end{equation}
where $\nu \in \mathbb{R}$ is the Lagrange multiplier. Denote the smallest eigenvalue of 
the spectral problem (\ref{spectral-problem}) by $\lambda_R \in (-\infty,1)$  
and the corresponding eigenfunction by $u_R \in \mathcal{H}^2_R$. Since $u_R(x) > 0$ and $\phi_R(x) > 0$ 
for all $x \in \mathbb{R}$, we have $\langle u_R, \phi_R \rangle_{\mathcal{H}_R} \neq 0$ (strictly positive). Therefore, the smallest eigenvalue 
of the constrained spectral problem (\ref{spectral-constrained}) is determined by the zero of the function 
$F(\lambda) : (-\infty,\lambda_R) \cup (\lambda_R,\lambda^*_R) \to \mathbb{R}$, where 
$$
F(\lambda) = \langle (\mathcal{L}_R - \lambda W_R)^{-1} W_R \phi_R, \phi_R \rangle_{\mathcal{H}_R} 
$$
and either $\lambda_R^* > 0$ is the second eigenvalue of the spectral problem (\ref{spectral-problem}) or $\lambda_R^* = 1$. 

By Theorem 2.7 in \cite{GP25}, the function $F(\lambda)$ is monotonically increasing in $(-\infty,\lambda_R) \cup (\lambda_R, \lambda^*_R)$ satisfying $F(\lambda) \to 0^+$ as $\lambda \to -\infty$ and a simple pole singularity as $\lambda \to \lambda_R$. Therefore, there are no zeros of $F(\lambda)$ for $\lambda \in (-\infty,\lambda_R)$ and if $\lambda_R \geq 0$, then 
the constraint $\langle u, \phi' \rangle_{L^2(\R)} = 0$ yields the coercivity estimate 
\begin{equation}
\label{tech-coerc}
Q_R(u) \geq \lambda_0 \| u \|^2_{\mathcal{H}_R}, \qquad \forall u \in \mathcal{H}^1_R(\R) : \;\;  \langle u, \phi' \rangle_{L^2(\R)} = 0,
\end{equation}
where $\lambda_0$ is either the lowest eigenvalue of the constrained spectral problem (\ref{spectral-constrained}) in $(0,1)$ or $1$ if no eigenvalues of 
(\ref{spectral-constrained}) in $(0,1)$ exists. 

If $\lambda_R < 0$ and $\lambda_R^* > 0$, then the lowest eigenvalue of the constrained spectral problem (\ref{spectral-constrained}) 
is strictly positive if and only if 
\begin{equation}
    \label{criterion}
F(0) = \langle (\mathcal{L}_R)^{-1} W_R \phi_R, \phi_R \rangle_{\mathcal{H}_R} = \langle W_R (\mathcal{L}_R)^{-1} W_R \phi_R, \phi_R \rangle_{L^2(\R)} < 0.
\end{equation}
Since $W_R \phi_R = \phi'(x)$ for every $x \in \mathbb{R}$, the criterion (\ref{criterion}) can be equivalently written as 
\begin{equation}
\label{E:Fzero-altform}
F(0) = \langle (\mathcal{L}_R)^{-1} \phi', \phi' \rangle_{L^2(\R)} < 0.
\end{equation}
Since $\lambda_R \to 0$ as $R \to +\infty$ implying $F(0) \to -\infty$ as $R \to +\infty$, the criterion (\ref{E:Fzero-altform}) is satisfied for sufficiently large $R$. This leads again to (\ref{tech-coerc}). The coercivity estimate (\ref{coercivity-R}) follows from (\ref{tech-coerc}) 
by G\"{a}rding's inequality.
\end{proof}

\subsection{Decomposition near the kink orbit}
\label{sec-2-3}

The constraint in (\ref{coercivity-minus}) can be satisfied by using the rotational parameter $\alpha$ in the orbit $\{ e^{-i \alpha} \phi \}_{\alpha \in \mathbb{R}}$. The constraint in (\ref{coercivity-R}) can be satisfied by using the translational parameter $\beta$ in the orbit $\{ \phi(\cdot - \beta)\}_{\beta \in \mathbb{R}}$. To incorporate both constraints, we define the decomposition relative to the two-dimensional orbit $\{ e^{-i \alpha} \phi(\cdot - \beta)\}_{\alpha, \beta \in \mathbb{R}}$ with the perturbation $u + i v$, where $u$ satisfies the constraint in (\ref{coercivity-R}) and $v$ satisfies the constraint in (\ref{coercivity-minus}), see Proposition \ref{prop-decomposition}.

We work in the energy space $\mathcal{E}$ in (\ref{energy-space}) equipped with the distance $\rho_R$ in (\ref{distance}). The following proposition shows that the element of $\mathcal{E}$ with small $\rho_R$ is bounded in $L^{\infty}(\mathbb{R})$.

\begin{proposition}
    \label{prop-supremum}
  If there is a small fixed $\epsilon_0 > 0$ such that 
  \begin{equation}
\label{bound-apriori}
\rho_R\left(\psi,\phi \right) \leq \epsilon^2_0,
\end{equation}
for every $\psi \in \mathcal{E}$, then $\psi \in L^{\infty}(\mathbb{R})$ and $\eta := |\psi|^2 - \phi^2$ satisfies 
  \begin{equation}
\label{bound-supremum}
\| \eta \|_{L^{\infty}(R,\infty)} \leq C \epsilon_0,
\end{equation}
for some $C > 0$.
\end{proposition}

\begin{proof}
Since $\rho_R(\psi,\phi) \geq \| \psi - \phi \|^2_{H^1(-\infty,R)}$, Sobolev's embedding of $H^1(-\infty,R)$ into $L^{\infty}(-\infty,R)$ implies that there is a constant $C > 0$ such that 
$$
\| \psi - \phi \|_{L^{\infty}(-\infty,R)} \leq C \sqrt{\rho_R(\psi,\phi)} \leq C \epsilon_0,
$$
where we have used (\ref{bound-apriori}) for the last inequality. 
If we represent $\psi = \phi + u + iv$ with some real $(u,v) \in \mathcal{H}_R^1 \times \mathcal{H}_R^1$, then 
$\eta = 2 \phi u + u^2 + v^2$ and the triangle inequality implies 
\begin{equation}
    \label{inequality-tech}
\| \eta \|_{L^{\infty}(-\infty,R)} \leq 2 \| u \|_{L^{\infty}(-\infty,R)} + \|u\|^2_{L^{\infty}(-\infty,R)} + \| v \|^2_{L^{\infty}(-\infty,R)} 
\leq 2 C \epsilon_0 + 2C^2 \epsilon_0^2 \leq 4 C \epsilon_0,
\end{equation}
where the last bound holds if $\epsilon_0 > 0$ is sufficiently small. 

To control $\| \eta \|_{L^{\infty}(R,\infty)}$, we use the bound 
$$
\| \eta \|^2_{L^{\infty}(R,\infty)} \leq \eta^2(R) + 2 \| \eta \|_{L^2(R,\infty)} \| \eta' \|_{L^2(R,\infty)}.
$$
Differentiating $\eta = 2 \phi u + u^2 + v^2$, we obtain 
\begin{align*}
    \| \eta' \|_{L^2(R,\infty)} &\leq 2 \left( \| \phi' u \|_{L^2(R,\infty)} + \| \phi u' \|_{L^2(R,\infty)} 
    + \| u u' \|_{L^2(R,\infty)} + \| v v' \|_{L^2(R,\infty)} \right) \\
    &\leq 2 \left( \| (1-\phi^2) u \|_{L^2(R,\infty)} + (1 + \| u \|_{L^{\infty}(R,\infty)}) \| u' \|_{L^2(R,\infty)} + \| v \|_{L^{\infty}(R,\infty)} \| v'\|_{L^2(R,\infty)} \right).
\end{align*}
Since $\rho_R(\psi,\phi) \geq \| \sqrt{1-\phi^2} (\psi-\phi) \|^2_{L^2(R,\infty)} + \| \psi' - \phi' \|^2_{L^2(R,\infty)} + \| \eta \|^2_{L^2(R,\infty)}$, we obtain 
\begin{align*}
\| \eta \|^2_{L^{\infty}(R,\infty)} &\leq \eta^2(R) + 4 \| \eta \|_{L^2(R,\infty)} (1 + \| u + iv \|_{L^{\infty}(R,\infty)}) (\| \sqrt{1-\phi^2} u \|_{L^2(R,\infty)} + \| u' + iv' \|_{L^2(R,\infty)}) \\
&\leq \eta^2(R) + 4 (1 + \| u + iv \|_{L^{\infty}(R,\infty)}) \rho_R(\psi,\phi).
\end{align*}
It follows from $(\phi + u)^2 + v^2 = \phi^2 + \eta$ that there is $C > 0$ such that 
$$
\| u + i v\|^2_{L^{\infty}(R,\infty)} \leq C (1 + \| \eta \|_{L^{\infty}(R,\infty)}).
$$
Since $|\eta(R)| \leq 4 C \epsilon_0$ by (\ref{inequality-tech}), the estimate on $\| \eta \|_{L^{\infty}(R,\infty)}$ can be closed as 
$$
\| \eta \|^2_{L^{\infty}(R,\infty)} \leq C (1 + \| \eta \|^{1/2}_{L^{\infty}(R,\infty)}) \epsilon_0^2.
$$
Since $\epsilon_0 > 0$ is small, this bound implies (\ref{bound-supremum}).
\end{proof}

The following proposition gives the decomposition near the two-parameter orbit 
of the kink solution with the profile $\phi$.

\begin{proposition}
    \label{prop-decomposition}
    Let $\psi \in C^0([-t_0,t_0],\mathcal{E})$ be a local solution of the NLS equation (\ref{NLS}) for some $t_0 > 0$ and assume that there is a small fixed $\epsilon_0 > 0$ such that 
    \begin{equation}
\label{bound-bigger}
	\inf_{\alpha,\beta \in \R} \rho_R\left(e^{i \alpha} \psi(t,\cdot + \beta),\phi \right) \leq \epsilon^2_0, \quad t \in [-t_0,t_0]. 
\end{equation}
Then, the solution can be uniquely represented as 
	\begin{equation}
	\label{decomposition}
	e^{i \alpha(t)} \psi(t,\cdot + \beta(t)) =  \phi + u(t,\cdot) + i v(t,\cdot), \quad \langle u, \phi' \rangle_{L^2(\R)} = 0, \quad \langle v, \phi \rangle_{\mathcal{H}_R} = 0, 
	\end{equation}
	for every $t \in [-t_0,t_0]$, with some real $\alpha(t),\beta(t) \in C^0([-t_0,t_0])$ and $u,v \in C^0([-t_0,t_0],\mathcal{H}^1_R)$.
\end{proposition}

\begin{proof}
In what follows, we fix $t \in [-t_0,t_0]$ and drop the time variable from the arguments of functions. To incorporate the two constraints in the decomposition (\ref{decomposition}), we introduce the function $F(\alpha,\beta;\psi) : \mathbb{R}^2 \times \mathcal{E} \cap \mathcal{H}_R \to \mathbb{R}^2$ given by 
    $$
    F(\alpha,\beta;\psi) = \left(   \begin{matrix} \langle {\rm Re}(e^{i \alpha} \psi(\cdot + \beta) - \phi), \phi' \rangle_{L^2(\R)} \\ 
    \langle  {\rm Im}(e^{i \alpha} \psi(\cdot + \beta) - \phi), \phi \rangle_{\mathcal{H}_R}
    \end{matrix} \right) 
    $$
    The infimum in (\ref{bound-bigger}) is attained if $\epsilon_0 > 0$ is sufficiently small, since $\alpha$ is defined on a compact interval $[0,2\pi]$  
    due to periodicity of $e^{i \alpha}$ and there is $C_{\infty} > 0$ such that 
    $$
    \lim_{\beta \to \pm\infty} \| \psi - \phi(\cdot - \beta) \|_{H^1(-\infty,R)} \geq C_{\infty}.
    $$
    Let $(\alpha_0,\beta_0) \in \mathbb{R}^2$ be the argument of the infimum in (\ref{bound-bigger}). By the Cauchy--Schwarz inequality, there is $C > 0$ such that $\| F(\alpha_0,\beta_0;\psi) \| \leq C \epsilon_0$ due to (\ref{bound-bigger}) and the exponential decay of $\phi'(x)$ and $1 - \phi^2(x)$ as $x \to +\infty$. 
    The Jacobian of $F(\alpha,\beta;\psi)$ in $(\alpha,\beta)$ is given by 
\begin{align*}   
    D F(\alpha,\beta;\psi) &= \left(   \begin{matrix} \langle -{\rm Im}(e^{i \alpha} \psi(\cdot + \beta)), \phi' \rangle_{L^2(\R)} 
    & \langle {\rm Re}(e^{i \alpha} \psi'(\cdot + \beta)), \phi' \rangle_{L^2(\R)}\\ 
    \langle  {\rm Re}(e^{i \alpha} \psi(\cdot + \beta)), \phi \rangle_{\mathcal{H}_R} & {\rm Im}(e^{i \alpha} \psi'(\cdot + \beta)), \phi \rangle_{\mathcal{H}_R}
    \end{matrix} \right). 
\end{align*}
By using (\ref{bound-bigger}) and the Cauchy--Schwarz inequality again, there is $C > 0$ such that 
\begin{align*}
    \left\| D F(\alpha_0,\beta_0;\psi) - \left(   \begin{matrix} 0 & \langle \phi', \phi' \rangle_{L^2(\R)}\\ 
    \langle  \phi, \phi \rangle_{\mathcal{H}_R} & 0 \end{matrix} \right) \right\| \leq C \epsilon_0,
\end{align*}
where we have used the fact that $\phi$ is real. Since $\| \phi' \|_{L^2(\R)}$ and $\| \phi \|^2_{\mathcal{H}_R}$ are finite and nonzero, 
the Jacobian of $F(\alpha,\beta;\psi)$ is invertible. Furthermore, $F(\alpha,\beta;\psi)$ is smooth with respect to its arguments. 
By the local inverse function theorem, for every $\psi \in \mathcal{E} \cap \mathcal{H}$ satisfying (\ref{bound-bigger}), 
there exists a unique $(\alpha,\beta) \in \R^2$ in a local neighborhood of $(\alpha_0,\beta_0) \in \R^2$ 
such that $F(\alpha,\beta;\psi) = 0$ and the decomposition (\ref{decomposition}) is proven.
\end{proof}

\subsection{Proof of Theorem \ref{th-main}}
\label{sec-2-4}

We are now ready to prove the orbital stability result for the kink with the profile $\phi$ in Theorem \ref{th-main}.

We start by considering a local solution  $\psi \in C^0([-t_0,t_0],\mathcal{E})$ of the NLS equation (\ref{NLS}) satisfying the bound (\ref{bound-bigger}) for some small fixed $\epsilon_0 > 0$. 
By Proposition \ref{prop-decomposition}, the solution can be written in the form (\ref{decomposition}). 
By using coercivity estimates (\ref{coercivity-minus}) and (\ref{coercivity-R}) in Lemmas \ref{lem-operator-minus} and \ref{lem-operator-R}, we get from energy conservation (\ref{energy}) and energy decomposition (\ref{energy-correct}):
\begin{align}	
E(\psi_0) - E(\phi) &= E(\phi + u + i v) - E(\phi) \notag \\
&\geq C (\| u' + i v'\|_{L^2(\R)} + \| u + iv \|^2_{\mathcal{H}_R} + \| \eta \|^2_{L^2(R,\infty)}) \notag \\
& \; - C \left( \| u + i v \|^3_{H^1(-\infty,R)} + \| u + i v \|^6_{H^1(-\infty,R)} + \| \eta \|_{L^{\infty}(R,\infty)} \| \eta \|^2_{L^2(R,\infty)} \right),
\label{energy-control}
\end{align}
for some positive constant $C$. It follows from (\ref{distance}) that 
$$
\| u + i v\|^2_{H^1(-\infty,R)} + \| \eta \|^2_{L^2(R,\infty)} \leq \rho_R(\phi + u + iv, \phi).
$$
By using Proposition \ref{prop-supremum}, as long as the bound (\ref{bound-apriori}) holds,  the bound (\ref{energy-control}) can be rewritten as 
$$
E(\psi_0) - E(\phi) \geq C \rho_R(\phi + u + iv, \phi), 
$$
by reducing the value of $C$ compared to (\ref{energy-control}). It follows from the energy decomposition (\ref{energy-decomp}) that 
$$
E(\psi_0) - E(\phi) \leq C \delta^2,
$$
for some $C > 0$ if $\rho_R(\psi,\phi) \leq \delta^2$ is small. This yields 
$$
\rho_R(e^{i \alpha(t)} \psi(t,\cdot + \beta(t)), \phi) \leq C \delta^2, \quad t \in [-t_0,t_0],
$$
for the solution $\psi \in C^0(\R,\mathcal{E})$ decomposed by (\ref{decomposition}). This defines $\epsilon^2 := C \delta^2 \in (0,\epsilon^2_0)$ 
and justifies the bound (\ref{bound-bigger}) beyond the time interval $[-t_0,t_0]$. Hence, the decomposition (\ref{decomposition}) and the energy bound (\ref{energy-control}) can be extended globally in time for every $t \in \R$. This completes the proof of the bound (\ref{bound-distance}) in Theorem \ref{th-main}.

\section{Numerical verification of the spectral criterion}
\label{sec-3}

Here, we describe a numerical study of the spectral problem \eqref{spectral-problem} and the verification of the sign condition \eqref{E:Fzero-altform}. The spectral problem $\mathcal{L}_R u = \lambda W_R u$  converts, upon setting $\tilde u = W_R^{1/2} u$, to the standard eigenvalue problem $W_R^{-1/2} \mathcal{L}_R W_R^{-1/2} \tilde u = \lambda \tilde u$. 
However, since $W_R(x) \to 0$ exponentially as $x\to +\infty$, the operator $W_R^{-1/2} \mathcal{L}_R W_R^{-1/2}$ is singular as $x\to +\infty$, which may lead to numerical instabilities. To avoid the singular behavior, we will describe a change-of-variable procedure.   

Recall the exact form of $\phi(x)$ given by \eqref{exact-kink}. We change variables from $x$ to $z$, where
\begin{equation} \label{E:xtoz}
z = \frac{x-\ln(2\cosh x)}{2} \quad \implies \quad \frac{dz}{dx} = \frac{1-\tanh x}{2} = 1-\phi^2(x)
\end{equation}
The transformation provides a one-to-one mapping $\mathbb{R} \ni x \to z \in (-\infty,0)$, which 
preserves the direction and satisfies the asymptotic expansions
\begin{equation}
\label{E:z-x-asymp}
z \sim \begin{cases} x - \frac12 e^{2x} & \text{as } x\to -\infty, \\ -\frac12 e^{-2x} &\text{as }x\to +\infty.
\end{cases}
\end{equation}
The transition point $R \in \mathbb{R}$ converts via \eqref{E:xtoz} to a corresponding point $z_R \in (-\infty,0)$.
The spectral problem \eqref{spectral-problem} converts to 
\begin{equation}
\label{E:LR-trans}
\tilde{\mathcal{L}}_R v = \lambda \tilde W_R v, \quad 
\tilde{\mathcal{L}}_R = - \frac{d}{dz} (1-\phi^2) \frac{d}{dz} + \tilde V_R, 
\end{equation}
where
\begin{equation*}
\tilde V_R(z) = \begin{cases} \frac{1-12\phi^2+15\phi^4}{1-\phi^2} & \text{if }z<z_R, \\
1-3\phi^2 & \text{if }z>z_R, \end{cases} \,, \qquad \tilde W_R(z) = \begin{cases} \frac{1}{1-\phi^2} & \text{if }z<z_R, \\ \frac{1}{1-\phi^2(R)} & \text{if }z>z_R. \end{cases}
\end{equation*}
Since $1-\phi^2 \sim -2z$ as $z\to 0^-$, the differential equation in \eqref{E:LR-trans} has a regular singular point at $z=0$, while 
$\tilde V_R(z)$ and $\tilde W_R(z)$ have nonzero constant asymptotics as $z\to 0^-$. The general asymptotic form  $\alpha x + \beta$ for solutions 
$u(x)$ of (\ref{spectral-problem}) as $x \to +\infty$ converts via (\ref{E:z-x-asymp}) to the behavior $-\frac12 \alpha \ln (-2z) + \beta$ for solutions $v(z)$ of (\ref{E:LR-trans}) as $z \to 0^-$.  We are, in particular, interested in the case $\alpha=0$ of solutions $v(z)$ with a finite limit as $z\to 0^-$, i.e., satisfying a Neumann boundary condition at this endpoint.



\begin{figure}[htb!]
    \centering
    \includegraphics[width=0.45\linewidth]{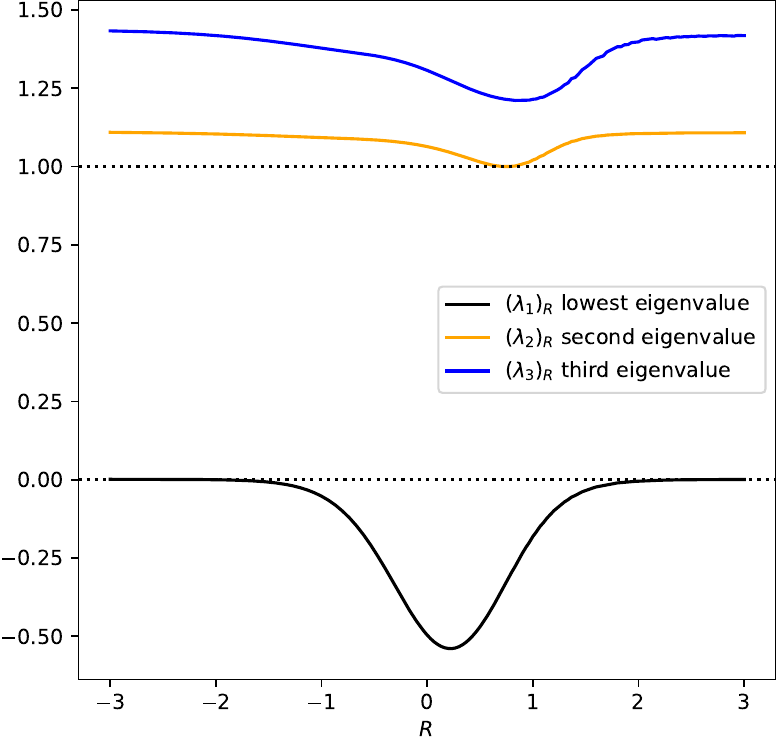}
    \includegraphics[width=0.45\linewidth]{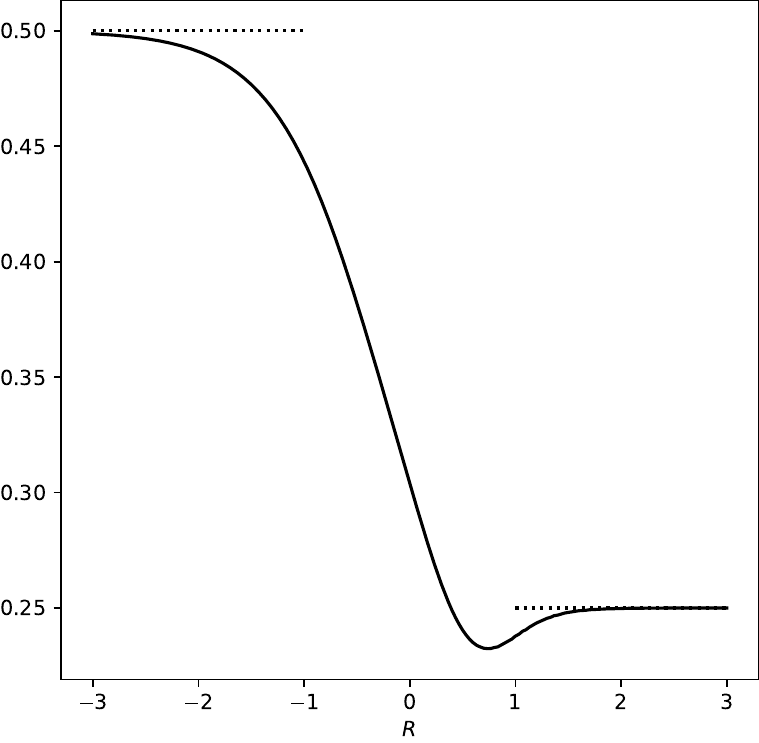}
    \caption{Left: plot of the lowest three numerically observed eigenvalues versus $R$, for the spectral problem \eqref{E:LR-trans}.  Right: plot of $(\lambda_1)_RF_R(0)$ versus $R$, showing that it remains positive and thus $F_R(0)<0$ for all $R$.}
    \label{F:evals_Fzero}
\end{figure}

Next, we solve the spectral problem \eqref{E:LR-trans} numerically using a finite difference method with equal grid spacing implementing Dirichlet boundary conditions at the left (truncated) endpoint $z = -10$ and Neumann boundary conditions at the right endpoint $z = 0$, with $N=20000$ grid points for a grid width $h=0.0005$.  Dirichlet boundary conditions were introduced at the left endpoint by assuming $v(-10-h)=0$ while Neumann boundary conditions were introduced at $z=0$ by assuming that $v(-h) = v(h)$ for a ghost point at $z=h$.  Sparse matrix representation was used with the functions \texttt{eigs}, \texttt{spsolve} in the python libraries \texttt{scipy.sparse.linalg}.

Figure \ref{F:evals_Fzero} (left panel) shows the lowest three eigenvalues of \eqref{E:LR-trans} versus parameter $R$.  The lowest eigenvalue 
$(\lambda_1)_R$ is found to be negative and the other two eigenvalues are found to be $\geq 1$, above the threshold for the continuous spectrum, 
within the range of error tolerance $\sim 10^{-3}$. Figure \ref{F:efuncs} shows the first three eigenfunctions for two values of $R$. We can see 
that the eigenfunctions for the second and third eigenvalues are long sinusoidal waves meeting the Dirichlet boundary condition at $z=-10$. 
This also suggests that these two eigenvalues become a part of the continuous spectrum in the limit as the left endpoint is sent to $-\infty$. 

\begin{figure}[htb!]
\centering
    \includegraphics[width=0.95\linewidth]{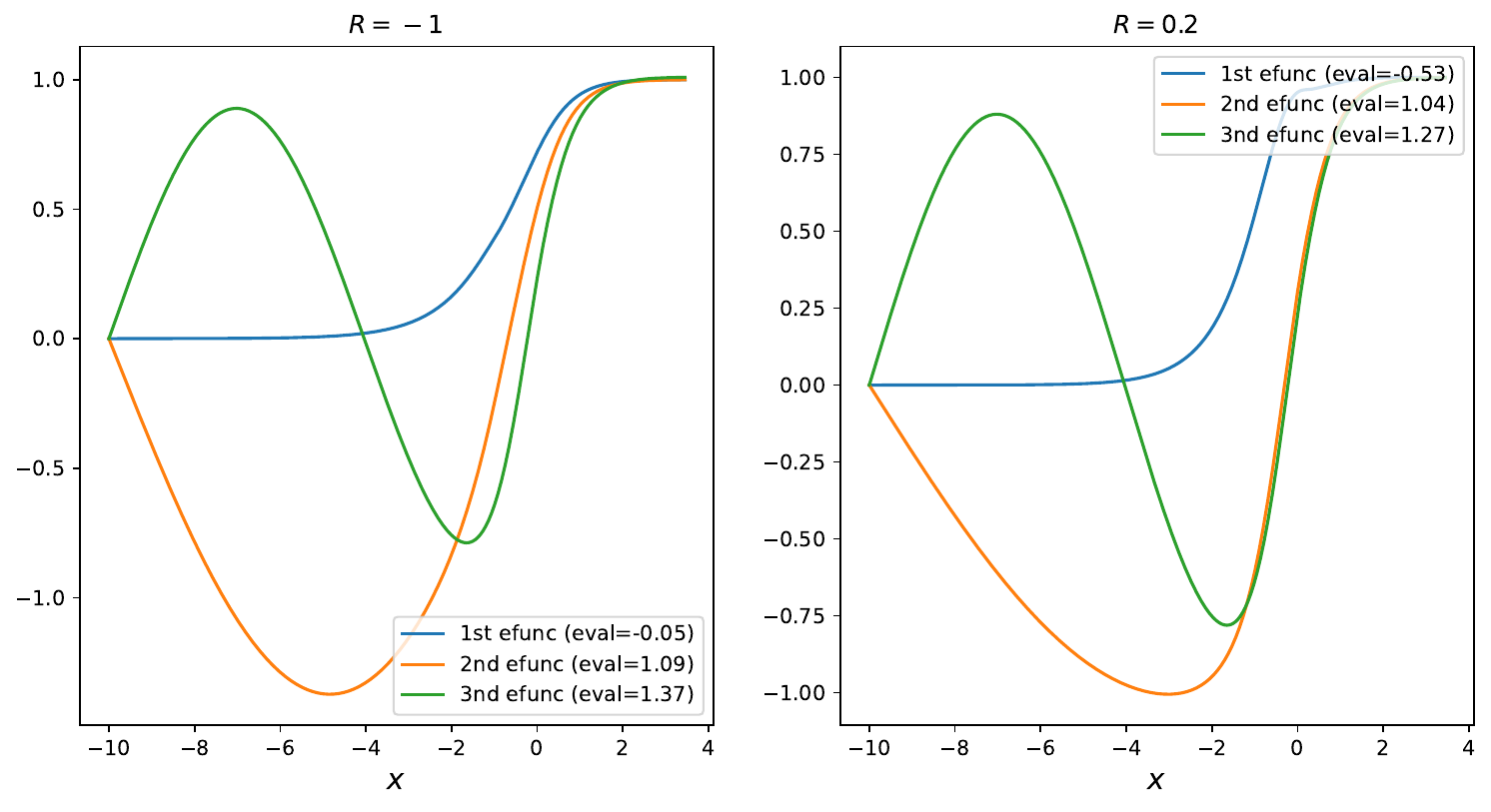}
    \caption{Plots of the first, second and third eigenfunctions obtained numerically for $R=-1$ (left) and $R=0.2$ (right). The horizontal axis is shown as $x$-values, although this computation was performed with the $z$ variable as described in the text.}
    \label{F:efuncs}
\end{figure}

To check \eqref{E:Fzero-altform}, we use 
$$
\frac{dz}{dx} \tilde{\mathcal{L}}_R = \mathcal{L}_R \quad \mbox{\rm and} \quad  \frac{\phi'}{1-\phi^2}=\phi,
$$ 
and obtain
$$
F_R(0) = \langle (\mathcal{L}_R)^{-1} \phi', \phi' \rangle_x = \langle (\tilde{\mathcal{L}}_R)^{-1} \phi, \phi \rangle_z
$$
The value of $F_R(0)$ for each $R$ is computed numerically using the sparse solver together with Simpson's rule for integration.  The results for $(\lambda_1)_RF_R(0)$ are shown in Figure \ref{F:evals_Fzero} (right) for a range of $R$.  It is apparent that $(\lambda_1)_RF_R(0)>0$ for all $R$ and there is a numerically convincing evidence of the limiting values
\begin{equation*}
\lim_{R\to -\infty} (\lambda_1)_RF_R(0) = \frac12 \,, \qquad \lim_{R\to +\infty} (\lambda_1)_RF_R(0) = \frac14.
\end{equation*}
Since $(\lambda_1)_R<0$, this shows that $F_R(0)<0$ for all $R$, completing the verification of \eqref{E:Fzero-altform}.

\section{Generalization for competing power nonlinearity}
\label{sec-4}

We now consider a generalized case of competing nonlinearities bearing
powers $p+1$ and $q+1$ (with the cubic-quintic representing the special case
of $p=2$ and $q=4$).
For the normalization purposes, similarly to the cubic--quintic NLS equation (\ref{NLS}), 
we set the frequency parameter to unity and consider the kink with the profile 
$\phi(x)$ satisfying $\phi(x) \to 0$ as $x \to -\infty$ and $\phi(x) \to 1$ as $x \to +\infty$. This gives uniquely coefficients 
of the competing nonlinearity with two powers $(p,q)$ satisfying $2 \leq p < q$ in the generalized NLS equation:
\begin{equation}
\label{NLS-power}
i \psi_t = \psi_{xx} - \psi + \frac{q (p+2)}{2 (q-p)} |\psi|^p \psi - \frac{p(q+2)}{2(q-p)} |\psi|^q \psi.
\end{equation}
The conserved energy function is given by
\begin{equation}
\label{energy-power}
E(\psi) = \int_{\mathbb{R}} \left[ |\psi_x|^2 + |\psi|^2 \frac{q(1 - |\psi|^p) - p (1 - |\psi|^q)}{q-p} \right] dx,
\end{equation}
and it is defined in the energy space 
\begin{equation}
\label{energy-space-power}
\mathcal{E} = \{ \psi \in H^1_{\rm loc}(\R) : \;\; 
\psi_x \in L^2(\R), \;\; \psi \sqrt{1 - |\psi|^p} \in L^2(\R), \;\; \psi \sqrt{1 - |\psi|^q} \in L^2(\R)\}.
\end{equation}
Local well-posedness of the initial-value problem for (\ref{NLS-power}) in the energy space $\mathcal{E}$ in (\ref{energy-space-power}) 
can be poven with the same method as in Appendix \ref{S:lwp}, although the global continuation of solutions may depend on the powers 
$(p,q)$ of the competing nonlinearity.

The kink profile is obtained from the zero level curve of the first-order invariant
\begin{equation*}
I(\phi,\phi') = (\phi')^2 - \phi^2 \frac{q(1 - \phi^p) - p (1 - \phi^q)}{q-p} = 0,
\end{equation*}
where $q > p$ and $q(1 - \phi^p) - p (1 - \phi^q) \geq 0$ on $(0,\infty)$ with the only minimum at $\phi = 1$.
The linearized system in (\ref{lin-NLS}) is associated with the linearized operators 
\begin{equation*}
\left\{ \begin{array}{l} \mathcal{L}_- = -\partial_x^2 + 1 - \frac{q (p+2)}{2 (q-p)} \phi^p + \frac{p(q+2)}{2(q-p)} \phi^q, \\
\mathcal{L}_+ = -\partial_x^2 + 1 - \frac{q (p+1) (p+2)}{2 (q-p)} \phi^{p} + \frac{p (q+1) (q+2)}{2(q-p)} \phi^{q}, \end{array} \right.   
\end{equation*}
which have the same spectral properties in $L^2(\mathbb{R})$ as in the cubic--quintic case since $\mathcal{L}_- \phi = 0$ 
and $\mathcal{L}_+ \phi' = 0$. By using (\ref{eta-variable}), we obtain the decomposition of the energy (\ref{energy-power}) in the form 
\begin{equation*}
E(\phi + u + i v) - E(\phi) = Q_-(u) + Q_-(v) + R(\eta),
\end{equation*}
where 
$$
Q_-(u) = \int_{\mathbb{R}} \left[ (u')^2 + u^2 - \frac{q (p+2)}{2 (q-p)} \phi^p u^2 + \frac{p(q+2)}{2(q-p)} \phi^q u^2 \right] dx
$$
and 
\begin{align*}    
R(\eta) &= \frac{p}{q-p} \int_{\mathbb{R}} \left[ (\phi^2 + \eta)^\frac{q+2}{2} - \phi^{q+2} - \frac{q+2}{2} \phi^q \eta \right] dx \\
& \quad - \frac{q}{q-p} \int_{\mathbb{R}} \left[ (\phi^2 + \eta)^\frac{p+2}{2} - \phi^{p+2} - \frac{p+2}{2} \phi^p \eta \right] dx.
\end{align*}
If $2 \leq p < q$, then the mapping $L^2(\mathbb{R}) \cap L^{\infty}(\mathbb{R}) \ni \eta \to R(\eta) \in \mathbb{R}$ is $C^2$ 
such that the decomposition can be expanded in the form:
\begin{align*}
R(\eta) = \frac{pq}{8(q-p)} \int_{\mathbb{R}} \left[ (q+2) \phi^{q-2} - (p+2) \phi^{p-2} \right] \eta^2 dx 
\left( 1 + \mathrm{o}(\| \eta \|_{L^{\infty}}) \right).
\end{align*}
By Proposition \ref{prop-supremum}, the remainder term in the decomposition of $R(\eta)$ is much smaller than the leading-order term 
if $\epsilon_0 > 0$ is small in the bound (\ref{bound-apriori}). 
By picking $R \in \mathbb{R}$ such that 
$$
\phi^{q-p}(R) > \frac{p+2}{q+2},
$$
the line $\mathbb{R}$ can be divided into $(-\infty,R)$ and $(R, \infty)$ with the same energy decomposition as in (\ref{energy-correct}). Lemmas 
\ref{lem-operator-minus} and \ref{lem-operator-R} as well as Proposition \ref{prop-decomposition} remain true so that the orbital stability theorem 
of the kink in $\mathcal{E} \cap \mathcal{H}_R$, where $\mathcal{H}_R$ is given by (\ref{Hilbert-space}) and $\mathcal{E}$ is given by (\ref{energy-space-power}), can be proven for the generalized NLS equation (\ref{NLS-power}) verbatim to the proof of Theorem \ref{th-main}. Note that the orbital stability 
result of Theorem \ref{th-main} only uses the local well-posedness of the initial-value problem in the energy space $\mathcal{E}$ since the solution 
is continued globally near the kink orbit by using the energy estimates.

\vspace{0.25cm}

{\it Acknowledgements.} This research was supported by the U.S. National Science Foundation under the award PHY-2408988 (PGK). This research was partly conducted while P.G.K. was 
visiting the Okinawa Institute of Science and
Technology (OIST) through the Theoretical Sciences Visiting Program (TSVP). 
This work was also 
supported by a grant from the Simons Foundation
[SFI-MPS-SFM-00011048, P.G.K]. P.G.K. is also grateful to Professor S.I. Mistakidis and
Drs. Katsimiga and Bougas for numerous constructive discussions on the subject of kink stability.

\appendix

\section{Proof of the local and global well-posedness in $\mathcal{E}$}
\label{S:lwp}

For completeness, we prove the 
well-posedness of the initial-value problem for the cubic--quintic NLS equation \eqref{NLS} 
in the energy space $\mathcal{E}$ given by (\ref{energy-space}). This can be obtained in the following steps.  First, we observe that $\mathcal{E} \subset L^\infty(\mathbb{R})$.  Indeed, 
$$(1-|\psi(x)|^2)^2 - (1-|\psi(x_*)|^2)^2 = -4 {\rm Re} \int_{x_*}^x (1-|\psi(y)|^2) \bar\psi(y) \psi'(y) \, dy$$
and thus
\begin{equation}\label{E:Linf_bound}\left| (1-|\psi(x)|^2)^2 - (1-|\psi(x_*)|^2)^2\right| \leq 4 \|\psi(1-|\psi|^2) \|_{L^2_{(x_*,x)}} \| \psi' \|_{L^2_{(x_*,x)}}
\end{equation}
If $\psi(1-|\psi|^2) \in L^2$, then there is a sequence $x_n \to +\infty$ such that either $\psi(x_n) \to 0$ or $|\psi(x_n)| \to 1$.  In either case, by applying \eqref{E:Linf_bound} with $x_*=x_n$ and sending $n\to \infty$, we obtain a uniform upper bound on $|\psi(x)|$.  This yields $\mathcal{E}\subset L^\infty(\R)$ and in fact further arguments using \eqref{E:Linf_bound} show that either $\lim\limits_{x\to \pm \infty} \psi(x)=0$ or $\lim\limits_{x\to \pm \infty} |\psi(x)|=1$.

Consider the Zhidkov space $X^1$, the Banach space of functions absolutely continuous on $\mathbb{R}$ with norm
$$
\| u \|_{X^1} = \| u \|_{L^\infty} + \| \partial_x u \|_{L^2}
$$
Since $\mathcal{E} \subset L^{\infty}(\R) $, we have that $\mathcal{E} \subset X^1$.  By \cite[Prop. I.2.7]{Zhidkov2001}, the Schr\"odinger semigroup $S(t): X^1\to X^1$ is uniformly bounded on any finite time interval $I$, the function $S(t)\phi:I \to X^1$ is continuous, and $\lim\limits_{t\to 0} S(t) \phi = \phi$, with the limit taken in $X^1$.  For the nonlinear term $f(\psi) = \psi(1-3|\psi|^2)(1-|\psi|^2)$, we have for some $C > 0$:
$$
\|f(\psi)\|_{L^\infty} \leq C \|\psi\|_{L^\infty} (1+\|\psi\|_{L^\infty})^4 \,, \qquad  \|\partial_x f(\psi)\|_{L^2} \leq C \|\psi_x \|_{L^2} (1+\|\psi\|_{L^\infty})^4.
$$
Therefore, a local solution can be obtained as a fixed point in $C^([0,t_0],X^1)$ via the contraction principle for some $t_0 > 0$. This local theory in the larger space $X^1$ can be restricted to $\mathcal{E}$, as follows.  If the initial condition satisfies $\psi(1-|\psi|^2)\in L^2(\R)$, then so does the local solution in $X^1$, by differentiation in time of  $\| \psi(1-|\psi|^2) \|_{L^2}^2$, substitution of the equation, integration by parts, and estimation in $X^1$. This provides a local solution $\psi(t,\cdot)$ in the energy space $\mathcal{E}$. Continuation of the local solution to the global solution $\psi \in C(\mathbb{R},\mathcal{E})$ is due to the global bounds following from the energy conservation.


\begin{thebibliography}{99}


\bibitem{Susanto2025} F. T. Adriano and H. Susanto, ``Maxwell fronts in the discrete nonlinear
Schr\"{o}dinger equations with competing nonlinearities", Stud. Appl. Math. (2025), to be published.

\bibitem{Alejo} M. A. Alejo and A. J. Corcho, ``Orbital stability 
	of the black soliton for the quintic Gross--Pitaevskii equation", 
Rev. Mat. Iberoam. {\bf 40} (2024) 1731--1780.

\bibitem{Baldelli2025NoDEA}
L.~Baldelli, ``On the cubic--quintic Schr\"odinger equation,''
Nonlinear Differ. Equ. Appl. {\bf 32} (2025) 105 (22 pages).

\bibitem{Saut} F. Bethuel, P. Gravejat, J.-C. Saut, and D. Smets, 
``Orbital stability of the black soliton for the Gross-Pitaevskii equation", Indiana Univ. Math. J. {\bf 57} (2008) 2611--2642.
    


\bibitem{Chiron} D. Chiron, ``Stability and instability for subsonic traveling waves of the nonlinear Schr\"{o}dinger equation in dimension one", 
Anal. PDE {\bf 6} (2013) 1327--1420.
    
\bibitem{CPP18} A. Contreras, D.E. Pelinovsky, and M. Plum, ``Orbital stability of domain walls in coupled Gross--Pitaevskii systems", 
SIAM J. Math. Anal. {\bf 50} (2018) 810--833.

\bibitem{crasovan} L.C. Crasovan,  B.A. Malomed, and D. Mihalache,  ``Spinning solitons in cubic-quintic nonlinear media.'', 
Pramana - J Phys {\bf 57} (2001) 1041--1059.

\bibitem{Gallo} L. Di Menza and C. Gallo, ``The black solitons of one-dimensional NLS equations", 
Nonlinearity {\bf 20} (2007) 461--496.

\bibitem{FalcaoFilho2013PRL}
E.~L.~Falc\~ao-Filho, C.~B.~de Ara\'ujo, G.~Boudebs, Herv\'e Leblond, and V.~Skarka,
``Robust two-dimensional spatial solitons in liquid carbon disulfide,''
Phys. Rev. Lett. \textbf{110}  (2013) 013901.

\bibitem{GPII} T. Gallay and D.E. Pelinovsky, ``Orbital stability in the cubic defocusing NLS equation. II. The black soliton",  J. Diff. Eqs. 
{\bf 258} (2015) 3639--3660.


\bibitem{GP25} A. Geyer and D. Pelinovsky, {\em Stability of nonlinear waves in Hamiltonian dynamical systems} Mathematical Surveys and Monographs {\bf 288} (AMS, Providence, 2025).

\bibitem{GS}  P. Gravejat and D. Smets. ``Asymptotic stability of the black soliton for the Gross-Pitaevskii equation''.
Proc. London Math. Soc. {\bf 111} (2015) 305--353.

\bibitem{Katsimiga2023PRA063308}
G.~C.~Katsimiga, S.~I.~Mistakidis, G.~N.~Koutsokostas, D.~J.~Frantzeskakis, R.~Carretero-Gonz\'alez, and P.~G.~Kevrekidis,
``Solitary waves in a quantum droplet-bearing system,''
Phys. Rev. A \textbf{107} (2023) 063308.

\bibitem{KillipOhPocovnicuVisan2017}
R.~Killip, T.~Oh, O.~Pocovnicu, and M.~Vi\c{s}an,
``Solitons and scattering for the cubic--quintic nonlinear Schr\"odinger equation on $\mathbb{R}^3$,''
Arch. Ration. Mech. Anal. \textbf{225} (2017) 469--548.



\bibitem{Luo2021FrontPhys}
Z.-H.~Luo, W.~Pang, B.~Liu, Y.-Y.~Li, and B.~A.~Malomed,
``A new form of liquid matter: Quantum droplets,''
Front. Phys. \textbf{16} (2021) 32201.


\bibitem{mistakidis2024generic}
S.I. Mistakidis, G. Bougas, G.C. Katsimiga, P.G. Kevrekidis,
``Generic transverse stability of kink structures in atomic and optical nonlinear media with competing attractive and repulsive interactions'',
Phys. Rev. Lett. \textbf{134} (2025) 123402.


\bibitem{Mistakidis2023PhysRep}
S.~I.~Mistakidis, A.~G.~Volosniev, R.~E.~Barfknecht, T.~Fogarty, T.~Busch, A.~Foerster, P.~Schmelcher, and N.~T.~Zinner,
``Few-body Bose gases in low dimensions---A laboratory for quantum dynamics,''
Phys. Rep. \textbf{1042} (2023) 1--108.


\bibitem{PP24} D.E. Pelinovsky and M. Plum, ``Stability of black solitons in optical systems with intensity-dependent dispersion", SIAM J. Math. Anal. {\bf 56} (2024) 2521--2568

\bibitem{Petrov2015}
D.~S.~Petrov,
``Quantum Mechanical Stabilization of a Collapsing Bose--Bose Mixture,''
Phys. Rev. Lett. \textbf{115} (2015)  155302.

\bibitem{PetrovAstrakharchik2016}
D.~S.~Petrov and G.~E.~Astrakharchik,
``Ultradilute low-dimensional liquids,''
Phys. Rev. Lett. \textbf{117} (2016)  100401.


\bibitem{Reyna2020PRA}
A.~S.~Reyna, H.~T.~M.~C.~M.~Baltar, E.~Bergmann, A.~M.~Amaral, E.~L.~Falc\~ao-Filho, P.-F.~Brevet, B.~A.~Malomed, and C.~B.~de Ara\'ujo,
``Observation and analysis of creation, decay, and regeneration of annular soliton clusters in a lossy cubic--quintic optical medium,''
Phys. Rev. A \textbf{102} (2020) 033523.

\bibitem{Zhidkov2001}
P.~E.~Zhidkov,
\emph{Korteweg de Vries and nonlinear Schr\"odinger equations: qualitative theory},
Lecture Notes in Mathematics, vol.~1756,
Springer, Berlin, 2001.


\end{thebibliography}
\end{document}